\documentclass[a4paper,12pt]{amsart}
\usepackage[T1]{fontenc}
\usepackage[english]{babel}
\usepackage[dvips]{graphicx}
\usepackage{amssymb,amsthm,amsmath}
\usepackage{pstricks}

\newcommand{\N}{\mathbb{N}}
\newcommand{\Z}{\mathbb{Z}}
\newcommand{\R}{\mathbb{R}}
\renewcommand{\H}{\mathbb{H}}

\newcommand{\Rn}{ {\mathbb{R}^n} }

\newcommand{\PS}{ {\Rn \setminus \{ 0 \}} }
\newcommand{\vol}{{\textnormal{vol}_k}}
\newcommand{\comment}[1]{}

\newtheorem{theorem}[equation]{Theorem}
\newtheorem{lemma}[equation]{Lemma}
\newtheorem{proposition}[equation]{Proposition}
\newtheorem{corollary}[equation]{Corollary}

\theoremstyle{definition}
\newtheorem{remark}[equation]{Remark}

\newtheorem{example}[equation]{Example}

\numberwithin{equation}{section}

\DeclareMathOperator{\diam}{diam}
\DeclareMathOperator{\spt}{spt}
\newcounter{minutes}
\divide\time by 60
\newcounter{hours}
\multiply\time by 60 \addtocounter{minutes}{-\time}
\title{\bf  Volume growth of quasihyperbolic balls}

\author[X.-H. Zhang]{Xiaohui Zhang}
\author[R. Kl\'en]{Riku Kl\'en}
\author[V. Suomala]{Ville Suomala}
\author[M. Vuorinen]{Matti Vuorinen}

\begin{document}

%%%%%%%%%%%%%%%%%%%%%%%%%%

%\thispagestyle{empty}
%\noindent {\bf\Large Publication III}

%\bigskip

%\noindent {\sc R. Kl\'en, V. Suomala, M. Vuorinen, and X.-H. Zhang}: {\it Volume growth of quasihyperbolic balls.} Available via arXiv:1208.5355 %[math.MG].

%\newpage
%\thispagestyle{empty}
%\mbox{}

%================

%\newpage
%\pagestyle{headings}
\setcounter{page}{1}

%%%%%%%%%%%%%%%%%%%%%%%%%%

\psset{linewidth=1pt}

\begin{abstract}
The purpose of this paper is to study the notion of the quasihyperbolic volume and to find growth estimates for the quasihyperbolic volume of balls in a domain $G \subset {\mathbb R}^n\,,$  in terms of the radii.

%It turns out that in the case of domains with Ahlfors regular boundaries, the rate of growth depends not merely on the radius but also
%on the metric structure of the boundary.
\end{abstract}

\maketitle



%%%%%%%%%%%%%%%%%%%%%%%%%%%%%%%%%%%%%%%%%%%%%%%%%%%%%%%%%%%%%%%%%%

%\comment{
%%%%%%%% BEGIN TIMESTAMP
\def\thefootnote{}
\footnotetext{ \texttt{\tiny File:~\jobname .tex,
          printed: \number\year-\number\month-\number\day,
         \thehours.\ifnum\theminutes<10{0}\fi\theminutes}
} \makeatletter\def\thefootnote{\@arabic\c@footnote}\makeatother
%%%%%%%% END TIMESTAMP
%}

%%%%%%%%%%%%%%%%%%%%%%%%%%%%%%%%%%%%%%%%%%%%
\section{Introduction}

Since its introduction three decades ago, the quasihyperbolic
metric has become a popular tool in many subfields of geometric
function theory, for instance, in the study of quasiconformal
maps of ${\mathbb R}^n, n\ge 2,$ and of Banach spaces \cite{gh,va}, analysis of metric spaces \cite{h} and hyperbolic type metrics \cite{himps}.
A natural question is whether or not and to what extent, the results
of hyperbolic geometry have counterparts for the quasihyperbolic
geometry. For instance in \cite{k} it was noticed that some
facts from hyperbolic trigonometry of the plane have counterparts
in the quasihyperbolic setup.

The purpose of this paper is to study the notion of the quasihyperbolic volume and to find growth estimates for the quasihyperbolic volume of balls in a domain $G \subset {\mathbb R}^n\,,$  in terms of the radii.
%It turns out that the rate of growth depends not merely on the radius but also on the metric structure of the boundary.
Below we give an explicit
growth estimate for the case of domains with Ahlfors regular boundary.

For a compact set with empty interior $E  \subset {\mathbb R}^n\,,$  and $0<s<t\,,$ we consider the layer sets
$E(s,t)=\{  z \in \Rn  : s\le d(z,E)\le t \} $
and relate their volume to the metric size of $E$ via its (Hausdorff)
dimension under the additional assumption that $E$ be Ahlfors $Q$-regular for some $0<Q<n$. It is practically equivalent to formulate this idea
in terms of the number of those Whitney cubes of the Whitney
decomposition of ${\mathbb R}^n \setminus E$ that meet $E(s,t)\,.$
This idea goes back to \cite{mv}. One of the key ideas of
\cite{mv} was to estimate the metric
size of $E$ from above in terms of the size of the layer sets. In Lemmas \ref{lemma:por} and \ref{lemma:reg}, we will refine some results of \cite{mv}.

In the last section, we estimate from below the growth of the quasihyperbolic volume of balls in the case when $G= \Rn\setminus
E$ is $\psi$-uniform and $E$ is Ahlfors regular. The main result of the paper is Theorem \ref{lowerbound} which
gives a lower bound for the volume growth.

%%%%%%%%%%%%%%%%%%%%%%%%%%%%%%%%%%%%%%%%%%%%
\section{Notation}

We assume $n \ge 2$ and use the notation $B^n(x,r)=\{z\in\Rn\, : \,|x-z|<r\}$ for the Euclidean ball, and its boundary is the sphere $S^{n-1}(x,r)=\partial B^n(x,r)$, where the center $x$ can be omitted if $x=0$.

The \emph{quasihyperbolic distance} between two points $x$ and $y$ in a proper subdomain $G$ of the Euclidean space $\Rn$, $n \ge 2$, is defined by
$$
  k_G(x,y) = \inf_{\alpha \in \Gamma_{xy}} \int_{\alpha}\frac{|dz|}{d(z)},
$$
where $d(z)=d(z,\partial G)$ is the (Euclidean) distance between the point $z \in G$ and the boundary of $G$ (denoted $\partial G$), and $\Gamma_{xy}$ is the collection of all rectifiable curves in $G$ joining $x$ and $y$.
Note that the quasihyperbolic metric in a half-space is just the hyperbolic metric, and that for a ball the two are bilipschitz equivalent.
Some basic properties of the quasihyperbolic metric can be found in \cite{vub}. In particular,  Martin and Osgood \cite[page 38]{mo} showed
that for $x,y\in\R^n\setminus\{0\}$ and $n\geq2$
\begin{equation}\label{PSformula}
k_{\R^n\setminus\{0\}}(x,y)=\sqrt{\theta^2+\log^2\frac{|x|}{|y|}},
\end{equation}
where $\theta=\measuredangle(x,0,y)\in[0,\pi]$.

For the volume of the unit ball in ${\mathbb R}^n$ we use the notation
 $\Omega_n:=m(B^n(0,1))=\pi^{n/2}/\Gamma((n/2)+1)$ where $m$ is the $n$-dimensional Lebesgue measure and $\Gamma$ stands for the
usual $\Gamma$-function, see \cite[Chapter 6]{as}. The surface area of the unit
sphere  is  $\omega_{n-1}: = m_{n-1}(S^{n-1}(0,1))=n\Omega_n  \,.$

The quasihyperbolic volume of a Lebesgue measurable set $A \subset G$ is defined by
$$
  \vol_{_G}(A) = \int_A \frac{dm(z)}{d(z)^n}\,.
$$
%where $m$ is the $n$-dimensional Lebesgue measure.
We also use the notation $\vol(A)$ for $\vol_{_G}(A)$ if the domain $G$ is clear from the context.

Below we use the Whitney decomposition of the complement of a closed set in $\Rn$. If $E \subset \Rn$ is non-empty and closed, then $\Rn \setminus E$ can be presented \cite[page 16]{s} as a union of closed dyadic cubes $Q_j^i$
\begin{equation}\label{wcube}
  \Rn \setminus E = \bigcup_{i \in \Z} \bigcup_{j=1}^{N_i}Q_j^i,
\end{equation}
where the cubes $Q_j^i$ have the following properties:

\begin{enumerate}
\item each $Q_j^i$ has faces parallel to the coordinate planes and edges of length $2^{-i}$,
\item the interiors of the cubes $Q_j^i$ and $Q_l^k$ are mutually disjoint provided $l\neq j$ or $i\neq k$,
\item the distance between the cube $Q_j^i$ and $E$ satisfies the following inequalities
\begin{equation}\label{wcestimate}
  2^{-i}\sqrt{n} \le d(Q_j^i,E) \le 2^{2-i}\sqrt{n}.
\end{equation}
\end{enumerate}
The decomposition (\ref{wcube}) is called the \emph{Whitney decomposition}, the cubes $Q_j^i$ are called \emph{Whitney cubes} and the set $\{ Q_j^i \colon j=1,\dots ,N_i \}$ is called the \emph{$i^{th}$ generation} of cubes.

For $0 < s \leq t < \infty$ and $E\subset\Rn$, we define
$$
E(s) = \{ x \in \Rn \colon  d(x,E) \leq s \}
$$
and
\begin{equation}\label{layerE}
E(s,t)=\{ x \in \Rn \colon  s\leq d(x,E) \leq t \} .
\end{equation}
We set $E(s,t)=\emptyset$ for $s>t$.
Because ${\rm diam}(Q_j^i) = \sqrt{n} 2^{-i}$ we have by (\ref{wcestimate})
\begin{equation}\label{Ek}
  Q_j^i \subset E(\sqrt{n}2^{-i},5\sqrt{n}2^{-i}).
\end{equation}

We say that a set $E\subset\Rn$ is $(p,r_0)$-\emph{uniformly porous} \cite{va-pro} if for all $x\in E$ and $0<r<r_0$ there is $y\in B(x,r)$ with $d(y,E)\ge\alpha r$. A set $E$ is uniformly porous, if it is $(p,r_0)$-uniformly porous for some $p,r_0>0$. We refer to $p,r_0$ and $n$ as the \emph{uniform porosity data}.

A set $E\subset\Rn$ is $Q$-\emph{regular} for $0<Q<n$, if there is a (Borel regular, outer-) measure $\mu$ with $\spt(\mu)=E$ and constants $0<\alpha\le\beta<\infty$ such that \[\alpha r^Q\le \mu(B^n(x,r))\le\beta r^Q,\text{ for all }x\in E\text{ and } 0<r<\diam(E)\,.\]
Here $\spt(\mu)$ denotes the smallest closed set with full $\mu$-measure. We refer to $n, Q,\alpha,\beta$ and $\diam(E)$ as the \emph{$Q$-regularity data}.

There is a close connection between the notions of uniform porosity and $Q$-regularity. Indeed, it is well known and easy to see (e.g. \cite[Lemma 3.12]{bhr}) that if $E$ is $Q$-regular for some $0<Q<n$, then it is uniformly porous. Conversely, if  $E$ is uniformly porous, then it is a subset of some $Q$-regular set for some $0<Q<n$. See \cite[Theorem 5.3]{jjkrs} for a more precise quantitative statement.

Let $\psi: [0,\infty)\to[0,\infty)$  be a homeomorphism with
$\psi(0) = 0$. A domain $G\subset\Rn$ is said to be $\psi$-\emph{uniform} \cite{vu} if
\begin{equation}
k_G(x, y)\leq\psi(|x-y|/\min\{d(x,\partial G), d(y,\partial G)\})
\end{equation}
for all $x, y\in G$. In particular, the domain $G$ is $L$-\emph{uniform} if $\psi(t)=L\log(1+t)$ with $L>1$.
 A domain is called \emph{uniform} if it is a $L$-uniform domain for some $L>1$. These domains have been studied by many authors, see \cite{go,klsv,mar,va88,vu}.

 %%%%%%%%%%%%%%%%%%%%%%%%%%%%%%%%%%%%%%%%%%%%
\section{Basic examples for the quasihyperbolic volume}

As mentioned in Introduction, for a given domain $G$ and a fixed point $x\in G\,,$  our goal is to find estimates for the asymptotic behaviour of $\vol(B_k(x,r))$  when
$$   B_k(x,r) = \{  z \in G:  k_G(x,z)<r \}$$
is the quasihyperbolic ball.
As an example, we first consider very regular domains for which the correct asymptotic behavior can be obtained by direct calculation.

%%%%%%%%%%%%%%%%%%%%%%%%%%%%%%
\subsection{\bf Unit ball.} {
Let us find the quasihyperbolic volume of $B_k(0,r)$ in the unit ball $G=B^n=B^n(0,1)$. Observe first that
$$
B_k(0,t) = B^n(0,1-e^{-t})\,.
$$
We need the following integral representation for the hypergeometric function: for $c>b>0$,
$$
 F(a,b;c;x)=\dfrac{1}{B(b,c-b)}\int_0^1 t^{b-1}(1-t)^{c-b-1}(1-xt)^{-a}dt
$$
where $B(b,c-b)=\Gamma(b)\Gamma(c-b)/\Gamma(c)$.
By definition
\begin{eqnarray}\label{kvolinBn}
  \vol(B^n(0,r)) & = & \int_{B^n(r)}\frac{dm(x)}{(1-|x|)^n} = \omega_{n-1} \int_0^r\frac{t^{n-1}}{(1-t)^n}dt \nonumber\\
   & = & \omega_{n-1} r^n \int_0^1 t^{n-1} (1-r t)^{-n}dt \nonumber\\
  & = & \frac{\omega_{n-1}r^n}{n}F(n,n;n+1;r)\nonumber\\
  & = & \frac{\omega_{n-1}r^n}{n(1-r)^{n-1}}F(1,1;n+1;r),
\end{eqnarray}
where the last equality follows from \cite[15.3.3]{as}
$$
F(a,b;c;x)=(1-x)^{c-a-b}F(c-a,c-b;c;x).
$$
Now $\vol(B_k(0,r)) =  \vol(B^n(0,1-e^{-r}))$ and hence for $0<r<\infty$
\begin{eqnarray*}
  \vol(B_k(0,r))&=&\frac{\omega_{n-1}\left(1-e^{-r}\right)^n}{n \left(1-(1-e^{-r})\right)^{n-1}}F(1,1;n+1;1-e^{-r})\\
  &\sim&\frac{\omega_{n-1}}{n-1}e^{(n-1)r}\quad (r\to\infty),
\end{eqnarray*}
since by \cite[15.1.20]{as}
$$
F(1,1;n+1;1)=\dfrac{\Gamma(n+1)\Gamma(n-1)}{\Gamma(n)^2}=\frac{n}{n-1}.
$$
}

\begin{proposition}\label{ratioinBn}
  For $s>0$, $\lambda>1$, $G = B^n$ and $E = \Rn\setminus G$
  $$
    \frac{\vol(E(s,\lambda s))}{\vol(E(s,\infty))} \to 1-\lambda^{1-n}
  $$
  as $s \to 0^+$.
\end{proposition}
\begin{proof}
  Let $0<s<1/\lambda$. Then the claim is equivalent to
  $$
    \frac{\vol(B^n(1-s)\setminus B^n(1-\lambda s)}{\vol(B^n(1-s))} \to 1-\lambda^{1-n}
  $$
  as $s \to 0^+$.

  By (\ref{kvolinBn})
  \begin{eqnarray*}
    & &\frac{\vol(B^n(1-s)\setminus B^n(1-\lambda s))}{\vol(B^n(1-s))}\\
    &= &1- \frac{\lambda^{1-n}(1-\lambda s)^nF(1,1;n+1;1-\lambda s)}{(1-s)^{n}F(1,1,n+1,1-s)}\\
    &\to&1-\lambda^{1-n}\quad(s\to0).
  \end{eqnarray*}
\end{proof}

\begin{remark}
For $G=B^n$, let $m_h(B^n(r))$ be the hyperbolic volume of the Euclidean ball $B^n(r)$. For a Lebesgue measurable set $A \subset B^n$ its hyperbolic volume is defined by
$$
m_h(A) = \int_A \frac{2^n \, dm(x)}{(1-|x|^2)^n} \,.
$$
It follows that
$$
\left(\dfrac{2}{1+r}\right)^n\vol(B^n(r))=\omega_{n-1}\int_0^r\dfrac{2^n t^{n-1}}{(1+r)^n(1-t)^n}dt
$$
$$
\leq \omega_{n-1}\int_0^r\dfrac{2^n t^{n-1}}{(1-t^2)^n}dt=m_h(B^n(r))
$$
and
$$
m_h(B^n(r))\leq\omega_{n-1}\int_0^r\dfrac{2^n t^{n-1}}{(1-t)^n}dt=2^n\vol(B^n(r)).
$$
It is easy to see that
$$m_h(B^n(r))=\omega_{n-1} 2^n r^n F(n/2,n;1+n/2;r^2)/n\,.$$
Corollary \ref{ratioinBn} is true for the hyperbolic metric with the same limiting value $1-\lambda^{n-1}$.
\end{remark}

%%%%%%%%%%%%%%%%%%%%%%%%%%%%%%%%%%%%%%%
\subsection{\bf Punctured space.} {\rm
We will consider the growth of quasihyperbolic volume in multiply connected domains. It can be shown \cite[Lemma 4.56, Theorem 4.57, Corollary 4.58]{k} that in $\R^2 \setminus \{ 0 \}$ the quasihyperbolic area of $B_k(x,r)$ is equal to $\pi r^2$, if $r \le \pi$, and $2\pi\sqrt{r^2-\pi^2}+2\pi^2 \arctan (\pi / \sqrt{r^2-\pi^2})$, if $r > \pi$. Therefore
$$2\pi\sqrt{r^2-\pi^2} \le \vol(B_k(x,r)) \le 4 \pi r$$
if $r > \pi$. We will now find similar lower and upper bounds for the quasihyperbolic volume of $B_k(x,r)$ in $\PS$ for $n > 2$.}

\begin{proposition}
  For $x \in \PS$, $n > 2$ and $r > \pi$
  $$
    2 \omega_{n-1} \sqrt{r^2-\pi^2} \le \vol(B_k(x,r)) \le 2 \omega_{n-1} r\,.
  $$
\end{proposition}
\begin{proof}
  Since $B_k(x,r)$ is invariant in the inversion in $S^{n-1}(|x|)$ we have $\vol(B_k(x,r)) = 2 \vol(B_k(x,r) \cap B^n(|x|) )$. By \cite[(3.9)]{vub} $B_k(x,r) \subset \Rn \setminus B^n(|x|e^{-r})$. Let $S=S^{n-1}(|x|e^{-\sqrt{r^2-\pi^2}})$, then by (\ref{PSformula}),  $\max\limits_{y\in S}k_{\PS}(x,y)=r$ implies
  $$B^n(|x|) \setminus B^n(|x|e^{-\sqrt{r^2-\pi^2}}) \subset B_k(x,r).$$
   Hence we have
  \begin{equation}\label{lowerPPestimate}
    \vol(B^n(|x|) \setminus B^n(|x|e^{-\sqrt{r^2-\pi^2}})) \le \vol(B_k(x,r) \cap B^n(|x|))
  \end{equation}
  and
  \begin{equation}\label{upperPPestimate}
    \vol(B_k(x,r) \cap B^n(|x|)) \le \vol(B^n(|x|) \setminus B^n(|x|e^{-r})).
  \end{equation}
  Next let us find the quasihyperbolic volume of annulus $E(a,b)=\{x\in\Rn\,|\,a<|x|<b\}$, $0 < a < b < \infty$. By definition
\begin{equation}\label{kvolinPS}
  \vol(E(a,b)) = \int_{E(a,b)}\frac{dm(x)}{|x|^n} = \omega_{n-1} \int_a^b\frac{t^{n-1}}{t^n}dt = \omega_{n-1} \log \frac{b}{a}.
\end{equation}
 Now the assertion follows from the equations (\ref{lowerPPestimate}) and (\ref{upperPPestimate}).
\end{proof}

\begin{corollary}
  For $s>0$, $\lambda>1$,  $G = B^n \setminus \{ 0 \}$ and $E = \Rn\setminus G$
  $$
    \frac{\vol(E(s,\lambda s))}{\vol(E(s,\infty))} \to 1-\lambda^{1-n}
  $$
  as $s \to 0$.
\end{corollary}
\begin{proof}
  Let $s < 1/\lambda$. Then $E(s,\lambda s) = E_1\cup E_2$ where $E_1=B^n(1-s) \setminus B^n(1-\lambda s)$ and $E_2=B^n(\lambda s) \setminus B^n(s)$. Then by (\ref{kvolinBn}) and (\ref{kvolinPS})
  \begin{eqnarray*}
    \vol_{_G}(E(s,\lambda s)) & = & \vol_{_{B^n}}(E_1) + \vol_{_\PS}(E_2)\\
    & = & \dfrac{\omega_{n-1}(1-s)^{n}}{ns^{n-1}}F(1,1;n+1,1-s)\\
    & & -\dfrac{\omega_{n-1}(1-\lambda s)^{n}}{n(\lambda s)^{n-1}}F(1,1;n+1,1-\lambda s)+\omega_{n-1}\log \lambda\\
    &\sim&\dfrac{\omega_{n-1}}{n-1}(1-\lambda^{1-n})s^{1-n}\quad (s\to0).
  \end{eqnarray*}
  Similarly $E(s,\infty) = B^n(1-s) \setminus B^n(s)=E_1'\cup E_2'$ where $E_1'=B^n(1-s)\setminus B^n(1/2)$ and $E_2'=B^n(1/2)\setminus B^n(s)$,  and
\begin{eqnarray*}
    \vol_{_G}(E(s,\infty)) & = & \vol_{_{B^n}}(E_1') + \vol_{_\PS}(E_2')\\
    & = & \dfrac{\omega_{n-1}(1-s)^{n}}{ns^{n-1}}F(1,1;n+1,1-s)\\
    & & -\dfrac{\omega_{n-1}}{2n}F(1,1;n+1,1/2)+\omega_{n-1}\log(1/2s)\\
    &\sim&\dfrac{\omega_{n-1}}{n-1}s^{1-n}+\omega_{n-1}\log(1/s)\quad (s\to0).
  \end{eqnarray*}
  Now we have
  $$
   \lim_{s\to0} \frac{\vol(E(s,\lambda s))}{\vol(E(s,\infty))}=
      \lim_{s\to0}\dfrac{\dfrac{\omega_{n-1}}{n-1}(1-\lambda^{1-n})s^{1-n}}
      {\dfrac{\omega_{n-1}}{n-1}s^{1-n}+\omega_{n-1}\log(1/s)}=1-\lambda^{1-n}.
  $$
\end{proof}

%%%%%%%%%%%%%%%%%%%%%%%%%%%%%%
\subsection{\bf Half space.} {\rm The next two propositions concern the
case of the half space, in which case the quasihyperbolic volume coincides
with the classical hyperbolic volume. The hyperbolic volume of  hyperbolic simplexes in the upper half space has been considered by J. Milnor in \cite{milnor}.}

\begin{proposition}
  For $x \in  \H^2$ and $r > 0$ we have
  \[
    \vol(B_k(x,r)) = 2 \pi (\cosh r -1).
  \]
\end{proposition}
\begin{proof}
  Let us choose $x=t e_2=(0,t)\in\H^2$ implying that
  \[
    B_k(x,r) = B^2(e_2 t \cosh r, t \sinh r).
  \]
  Now
  \begin{eqnarray*}
    \vol(B_k(x,r)) &=& \int_{t\cosh r-t\sinh r}^{t\cosh r+t\sinh r}\frac{\sqrt{t^2 \sinh^2 r-(\cosh r-h)^2}}{h^2}dh\\
     &=& 2 \pi (\cosh r -1).
  \end{eqnarray*}
\end{proof}

\begin{proposition}
  For $x \in  \H^3$ and $r > 0$ we have
  \[
    \vol(B_k(x,r)) = \pi (\sinh (2r) -2r).
  \]
\end{proposition}
\begin{proof}
  Let us choose $x=t e_3=(0,0,t)\in\H^3$ implying that
  \[
    B_k(x,r) = B^3(e_3 t \cosh r,t \sinh r).
  \]
  Now
  \begin{eqnarray*}
    \vol(B_k(x,r)) &=& \int_{t\cosh r-t\sinh r}^{t\cosh r+t\sinh r}\frac{\pi (t^2 \sinh^2 r-(\cosh r-h)^2)}{h^3}dh\\
    &=& \pi (\sinh (2r) -2r).
  \end{eqnarray*}
\end{proof}

%%%%%%%%%%%%%%%%%%%%%%%%%%%%%%%%%%%%%%%%%%%%
\section{Number of Whitney cubes}\label{sec:Whitney}

Recall that $N_i$ denotes the number of the  $i^{th}$ generation Whitney cubes of $\Rn\setminus E$. We also define the number
\begin{equation}\label{eqn:Ntilde}
\widetilde{N}_i=N_{i+k_0}+\ldots+N_{i+k_1},
\end{equation}
where $k_0$ is the smallest integer satisfying $8\sqrt{n}2^{-k_0}\le 1$ and $k_1$ is the largest integer with $80\sqrt{n} 2^{-k_1}\ge 1$.

\begin{lemma}\label{lemma:por}
If $E\subset \Rn$ is compact and uniformly porous, then there are constants $0<c<C<\infty$ and $i_0\in\N$ (depending only on the uniform porosity data) such that for all $i\ge i_0$, then it holds that
\begin{equation*}
c 2^{in}m(E(2^{-i}))\le\widetilde{N}_i\le C 2^{in} m(E(2^{-i}))\,.
\end{equation*}

\end{lemma}

\begin{proof}
Let $0<\lambda<1<\Lambda<\infty$. Choose $0<p<(1+\lambda)/(2\Lambda)$ and $r_0>0$ such that $E$ is $(2p,r_0)$-uniformly porous. Given $r>0$ such that $R=\tfrac{1}{2p}(1+\lambda) r<r_0$, let $B_1,\ldots,B_N$ be a maximal collection of disjoint balls of radius $R$ centered at $E$. Then by elementary geometry
\begin{equation}\label{uniformR}
N \Omega_n R^n\le m(E(R))\le 3^n N \Omega_n R^n\,.
\end{equation}

For each $B_i$, we can find $y_i\in \tfrac12B_i$ with $d(y_i,E)\ge p R=\tfrac12(1+\lambda)r$. Moreover (since we live in the Euclidean space), we can in fact find such $y_i$ with $d(y_i,E)=\tfrac12(1+\lambda)r$. Then $B(y_i,\tfrac12(1-\lambda)r)\subset E(r)\setminus E(\lambda r)$ and thus, combined with \eqref{uniformR},
\begin{align*}
m(E(r)\setminus E(\lambda r))&\ge N \Omega_n \left(\frac{1-\lambda}{2}\right )^{n} r^n\ge\left(\frac{p(1-\lambda)}{3(1+\lambda)}\right)^n 3^n N\Omega_n R^{n}\\
&\ge \left(\frac{p(1-\lambda)}{3(1+\lambda)}\right)^n m(E(\Lambda r))\,,
\end{align*}
recall that $R\ge \Lambda r$.

Applying the above estimate with $r=2^{-i-3}$, $\lambda=\tfrac12$ and $\Lambda=8$ and using \eqref{Ek} implies
\begin{align*}
m(E(2^{-i}))&\le c m(E(2^{-i-3})\setminus E(2^{-i-4}))\\
&\le c\sum_{j=k_0}^{k_1}N_{i+j} 2^{-(i+j)n}\le c m(E(2^{-i}))
\end{align*}
where $c<+\infty$ depends only on $n$ and $p$.
\end{proof}

We also need the following simple fact

\begin{lemma}\label{lemma:reg}
If $E\subset\Rn$ is compact and $Q$-regular for some $0<Q<n$, then there are constants $0<c<C<\infty$ (depending only on the $Q$-regularity data) such that
$c r^{n-Q}\le m(E(r))\le C r^{n-Q}$ for $0<r<\diam(E)$.
\end{lemma}

\begin{proof}
Let $r\in(0,\diam(E))$. Let $B_1,\ldots, B_{N}$ be a maximal collection of disjoint balls each of radius $r$ centered at points in $E$. Then, as in \eqref{uniformR},
\begin{equation*}
\bigcup_{i=1}^NB_i\subset E(r)\subset\bigcup_{i=1}^N3B_i\,,
\end{equation*}
and in particular,
\begin{equation}\label{eq:3r}
N \Omega_n r^n\le m(E(r))\le N \Omega_n 3^n r^n\,.
\end{equation}
Using the $Q$-regularity, we have
\begin{equation}\label{eq:mu}
N \alpha r^Q\le \sum_i\mu(B_i)\le \mu(E) = \mu(E(r)) \le\sum_i\mu(3 B_i)\le N \beta 3^Q r^Q\,.
\end{equation}
The claim follows by combining \eqref{eq:3r} and \eqref{eq:mu}.
\end{proof}

Putting the two previous Lemmas together, we obtain:

\begin{corollary}\label{cor:reg}
If $E\subset \Rn$ is compact and $Q$-regular for some $0<Q<n$, then there are constants $0<c<C<\infty$  and $i_0\in\N$ (depending only on the $Q$-regularity data) such that for all $i\ge i_0$,
\begin{equation*}
c 2^{iQ}\le\widetilde{N}_i\le C 2^{iQ}\,.
\end{equation*}
\end{corollary}

\begin{remark}
1. The porosity and regularity assumptions in the above results are essential because without them, we cannot guarantee that $m(E(2r)\setminus E(r))$ is comparable to $m(E(r))$ for small $r>0$.

2. In \cite{klv} the authors considered the relations between the number of Whitney balls and Minkowski dimension, in an Ahlfors regular space, which also can be reformulated equivalently in terms of the measures of the layer sets. They also showed a close connection between the 'surface area' of the layer sets and the number of Whitney balls.
\end{remark}

%%%%%%%%%%%%%%%%%%%%%%%%%%%%%%%%%%%%%%%%%%%%
%%%%%%%%%%%%%%%%%%%%%%%%%%%%%%%%%%%%%%%%%%%%

%%%%%%%%%%%%%%%%%%%%%%%%%%%%%%%%%%%%%%%%%
%%%%%%%%%%%%%%%%%%%%%%%%%%%%%%%%%%%%%%%%%

%%%%%%%%%%%%%%%%%%%%%%%%%%%%%%%%%%%%%%%%%%%%%%%%%%%%%%%%%%%%%%%
%%%%%%%%%%%%%%%%%%%%%%%%%%%%%%%%%%%%%%%%%%%%%%%%%%%%%%%%%%%%%%%

\section{Growth of the quasihyperbolic volume}

Our next goal is to prove growth estimates of the above type
for a larger class of domains, the so-called $\varphi$-uniform
domains. We first observe that the
quasihyperbolic volume of any Whitney cube is essentially a constant.

\begin{lemma}\label{lem:vol4wcube}
  Let $E$ be a closed subset of $\Rn$ and let $G$ be a component of $\Rn\setminus E$. Then for the Whitney cubes of $\Rn \setminus E$ contained in $G$ we have
\[
    2^{-2n}n^{-n/2} \le \vol(Q_j^i) \le n ^{-n/2}.
\]
\end{lemma}

\begin{proof}
  Since $m(Q_j^i) = 2^{-in}$ we have by (\ref{wcestimate})
\begin{equation*}
    \vol(Q_j^i) \ge \frac{2^{(i-2)n}}{\sqrt{n}^n}2^{-in} = 2^{-2n}n^{-n/2},
\end{equation*}
and
$$
\vol(Q_j^i) \le \frac{2^{in}}{\sqrt{n}^n}2^{-in} = n^{-n/2}.
$$
\end{proof}

\begin{theorem}\label{upperbound}
Let $G$ be a proper subdomain of\, $\Rn$ with compact and $Q$-regular boundary for some $Q\in(0,n)$.
For each $x\in G$, there exists a constant $C<\infty$ which depends only on the $Q$-regularity data and on $d(x)$ such that for all sufficiently large $r$, we have
$$
\vol(B_k(x,r))\leq C e^{Qr}.
$$
\end{theorem}

\begin{proof}
Let $E=\partial G$.
We have
$$B_k(x,r)\subset\{z\in G: e^{-r}d(x)<d(z)<e^{r}d(x)\}=:D$$
and
$$Q_j^i\subset E(2^{-i}\sqrt{n},5\cdot2^{-i}\sqrt{n}).$$
If $5\cdot2^{-i}\sqrt{n}<e^{-r}d(x)$ or $2^{-i}\sqrt{n}>e^{r}d(x)$, i.e.
$$i>\dfrac{r+\log(5\sqrt{n}/d(x))}{\log2}=:K,\quad\mbox{or}\quad i<-\dfrac{r-\log(\sqrt{n}/d(x))}{\log2}=:K'$$
then $Q_j^i\cap D=\emptyset,$ and thus
\begin{equation}\label{eq:ballcontained}
B_k(x,r)\subset\bigcup_{i={[K']}}^{[K]}\bigcup_{j=1}^{N_i}Q_j^i.
\end{equation}
Note that the inequality $N_i\leq M(\diam(E))2^{ni}$ always holds for all $i\in \N$ \cite{mv} and therefore $N_i\leq C(\diam(E),i_0)$ for $i\leq i_0$, where $i_0$ is as in Lemma \ref{lemma:por}.
Hence for all sufficiently large $r$, by \eqref{eq:ballcontained} and Lemma \ref{lem:vol4wcube}, we have
$$\vol(B_k(x,r))\leq\sum_{i={[K']}}^{[K]}\sum_{j=1}^{N_i}\vol(Q_j^i)\leq n^{-n/2}\sum_{i={[K']}}^{[K]}N_i.$$
Using the estimates for the number of Whitney cubes from Lemma \ref{lemma:por} and for the measure of level sets from Lemma \ref{lemma:reg}, we get
\begin{eqnarray*}
\vol(B_k(x,r))&\leq&n^{-n/2}(i_0-[K'])C(\diam(E),i_0)\\
              & &+n^{-n/2}\sum_{i=i_0}^{[K]}c2^{in}m(E(2^{-i}))\\
              &\leq&n^{-n/2}(i_0-[K'])C(\diam(E),i_0)\\
              & &+n^{-n/2}\sum_{i=i_0}^{[K]}c2^{in}c'(2^{-i})^{n-Q}\\
              &\leq&n^{-n/2}(i_0-[K'])C(\diam(E),i_0)+\dfrac{cc'n^{-n/2}}{2^{Q}-1}2^{KQ}\\
              &\leq&C e^{Qr},
\end{eqnarray*}
where the last inequality holds for sufficiently large $r$ by the definition of $K$ and $K'$.
\end{proof}

\comment{

\begin{lemma}\label{thm:hausdorff}
Let $\emptyset\neq E\subset\Rn$ be an $(\alpha,r_0)$-uniformly porous compact set such that $G=\Rn\setminus E$ is a  $\psi$-uniform domain and let $x\in G$. Then there exists $r_1=r_1(x)>0$ such that
  $$
    d(w,B_k(x,r))\leq\dfrac{{\rm diam}(\partial G)+d(x)+r_0}{\alpha\psi^{-1}(r)}
  $$
  for all $w\in \partial G$ and $r\geq r_1$.
\end{lemma}

\begin{proof}
  Let $w\in\partial G$. Since $\partial G$ is uniformly $(\alpha,r_0)$-porous, for each $s\in(0,r_0]$ there exists $y\in G\cap B^n(w,s)$ with $d(y)>\alpha s$. By $\psi$-uniformity of $G$,
  $$    k_G(x,y)\leq\psi\left(\dfrac{|x-y|}{\min\{d(x),d(y)\}}\right)\leq\psi\left(\dfrac{{\rm diam}(\partial G)+d(x)+r_0}{\min\{d(x),\alpha s\}}\right),
  $$
  which implies
\[d\left(w,B_k\left(x,\psi\left(\dfrac{{\rm diam}(\partial G)+d(x)+r_0}{\min\{d(x),\alpha s\}}\right)\right)\right)\le s\,.\]
Substituting $s=(\diam(\partial G)+d(x)+r_0)/(\alpha\psi^{-1}(r))$ this reads
\[d(w,B_k(x,r))\le\frac{\diam(\partial G)+d(x)+r_0}{\alpha\psi^{-1}(r)}\]
for $r\ge\psi\left(\left(\diam(\partial G)+d(x)+r_0\right)/\min\{d(x),\alpha r_0\}\right)=:r_1.$
\end{proof}

}%end of comment

\begin{lemma}\label{thm:hausdorff}
Let $\emptyset\neq E\subset\Rn$ be a  compact set such that $G=\Rn\setminus E$ is a  $\psi$-uniform domain and let $x\in G$. Then for all $r>0$,
$$E\left(\frac{{\rm diam}(E)+2d(x)}{\psi^{-1}(r)},d(x)\right)\subset B_k(x,r).$$
\end{lemma}

\begin{proof}
We only need to consider for $r\geq\psi(({\rm diam}(E)+2d(x))/d(x))$. Otherwise, it is trivial.
For arbitrary
$$y\in E\left(\frac{{\rm diam}(E)+2d(x)}{\psi^{-1}(r)},d(x)\right),$$
we have that
\begin{equation}\label{eqn:dy}
d(y)<d(x)\qquad\mbox{and}\qquad \psi\left(\dfrac{{\rm diam}(E)+2d(x)}{d(y)}\right)<r.
\end{equation}
By \eqref{eqn:dy} and the $\psi$-uniformity of $G$ we get
\begin{align*}
k_G(x,y)&\leq\psi\left(\dfrac{|x-y|}{\min\{d(x),d(y)\}}\right)\\
        &\leq\psi\left(\dfrac{{\rm diam}(E)+d(x)+d(y)}{\min\{d(x),d(y)\}}\right)\\
        &\leq\psi\left(\dfrac{{\rm diam}(E)+2d(x)}{d(y)}\right)\\
        &< r,
\end{align*}
which implies the desired inclusion.
\end{proof}

\begin{theorem}\label{lowerbound}
Let $E\subset\Rn$ be a compact $Q-$regular set with $0<Q<n$ such that $G=\Rn\setminus E$ is a $\psi-$uniform domain. Then for each $x\in G$ there exist $c>0$ and $r_1>0$ such that
$$
  \vol(B_k(x,r))\geq c (\psi^{-1}(r))^Q\,,
$$
for all $r>r_1$, where $c$ and $r_1$ depend only on the $Q$-regularity data and $d(x)$.
\end{theorem}

\begin{proof}
Let $K$ be the largest integer such that
\begin{equation}\label{eqn:lower}
2^{-(K+k_1)}\sqrt{n}\ge\frac{{\rm diam}(E)+2d(x)}{\psi^{-1}(r)},
\end{equation}
and let
$$r_1=\psi\left(({\rm diam}(E)+2d(x))\max\left\{\frac{10\cdot 2^{k_1-k_0}}{d(x)},\frac{2^{i_0+k_1}}{\sqrt{n}}\right\}\right),$$
where $i_0$ is as in Lemma \ref{lemma:por}, and $k_1$ and $k_0$ are as in \eqref{eqn:Ntilde}.
Then for all $r>r_1$, we have that
\begin{equation}\label{eqn:C}
2^{-K}<\frac{C}{\psi^{-1}(r)},\quad C:=\frac{2^{k_1+1}({\rm diam}(E)+2d(x))}{\sqrt{n}},
\end{equation}
and
\begin{equation}\label{eqn:upper}
5\sqrt{n}2^{-(K+k_0)}\leq d(x).
\end{equation}
For $r>r_1$, it follows from \eqref{Ek}, \eqref{eqn:lower}, \eqref{eqn:upper}, and Lemma \ref{thm:hausdorff}
that  for $K+k_0\leq m\leq K+k_1$, all the $m^{th}$-generation Whitney cubes belong to $B_k(x,r)$. Together with Corollary \ref{cor:reg} and Lemma \ref{lem:vol4wcube} this yields for large values of $r$ the required estimate
\[\vol(B_k(x,r))\ge c_0 \widetilde{N}_K\ge c_1 2^{KQ}\ge c(\psi^{-1}(r))^Q\,,\]
where the last inequality follows from \eqref{eqn:C}, and  the constant $c>0$ only depends on the $Q$-regularity data and $d(x)$.
\end{proof}

{\rm
By combining Theorem \ref{upperbound}, Theorem \ref{lowerbound} and the definition of uniform domain we get the following corollary.}

\begin{corollary}\label{cor:cor}
Let $E\subset \Rn$ be a compact $Q-$regular set with $0<Q<n$ such that $G=\Rn\setminus E$ is a uniform domain with the uniformity constant $L>1$. Then for each $x\in G$ and sufficiently large $r>0$,
$$
  c e^{Qr/L}\leq\vol(B_k(x,r))\leq C e^{Qr},
$$
where $C<\infty$  depends only on the $Q$-regularity data and $L$ and $c>0$ depends only on the $Q$-regularity data, $L$, and $d(x)$.
\end{corollary}

\begin{example}
Let $E\subset \Rn$ be a self-similar set whose complement is a uniform domain. For instance, $E$ can be a Cantor set on a hyperplane, the $\tfrac14$-Cantor set in the plane, or more generally any self-similar set satisfying the strong separation condition (See e.g. \cite{f} for the definitions). Then $E$ is $Q$-regular in its dimension and Corollary \ref{cor:cor} can be applied.
\end{example}

\begin{remark}
Although we stated the Theorem \ref{lowerbound} for unbounded domains, essentially the same proof works for all $\psi$-uniform domains $G\subset\Rn$ whose boundary is $Q$-regular and uniformly porous in $G$. The last assumption means that that the porosity holes in the definition of the uniform porosity lie completely inside $G$. For instance, in $\R^2$, $\partial G$ can be the Von-Koch snowflake curve. More generally, the boundary could be the union of a finite number of such $Q_i$-regular pieces, and in this case we could apply the lower bound with the exponent $r\min\{Q_i\}/L$ and the upper bound with the exponent $r\max\{Q_i\}$.
\end{remark}

{\rm It is not known what is the best possible lower bound in Corollary \ref{cor:cor}. It is an interesting open problem if the exponent $Qr/L$ could be replaced by $cr$ for some uniform constant $c>0$ independent of $L$.}

\subsection*{Acknowledgments}
{\rm  The research of Matti Vuorinen was supported by the Academy of Finland, Project 2600066611. Xiaohui Zhang is indebted to the Finnish National Graduate School of Mathematics and its Applications for financial support. The research of Ville Suomala was partially supported by the Academy of Finland.}

\medskip

%%%%%%%%%%%%%%%%%%%%%%%%%%%%%%%%%%%%%%%%%%%%%%%%%%%%%%%%%%%%%

%%%%%%%%%%%%%%%%%%%%%%%%%%%%%%%%%%

\bigskip

{\rm\small

({\it Xiaohui Zhang}) {\sc Department of Mathematical Sciences, Zhejiang Sci-Tech University, 310018 Hangzhou, China}

{\it E-mail address}: \verb"xiaohui.zhang@zstu.edu.cn"

\medskip

({\it Riku Kl\'en, Matti Vuorinen}) {\sc Department of Mathematics and Statistics, University of Turku, 20014 Turku, Finland}

{\it E-mail address}: \verb"riku.klen@utu.fi, vuorinen@utu.fi"

\medskip

({\it Ville Suomala}) {\sc Department of Mathematical sciences, P.O Box 3000, FI-90014 University of Oulu, Finland}

{\it E-mail address}: \verb"ville.suomala@oulu.fi"

}

\end{document}